\documentclass[11pt,a4paper]{article}
\usepackage[T1]{fontenc}
\usepackage[utf8]{inputenc}
\usepackage{lmodern,amsmath,amssymb,amsthm,mathtools,mathrsfs}
\usepackage[margin=28mm]{geometry}
\usepackage{microtype,enumitem,xcolor}
\usepackage[colorlinks=true,linkcolor=blue!45!black,citecolor=blue!45!black,urlcolor=blue!45!black]{hyperref}
\usepackage{fancyhdr}
\newtheorem{theorem}{Theorem}[section]
\newtheorem{proposition}[theorem]{Proposition}
\newtheorem{lemma}[theorem]{Lemma}
\newtheorem{corollary}[theorem]{Corollary}
\theoremstyle{definition}\newtheorem{definition}[theorem]{Definition}
\theoremstyle{remark}\newtheorem{remark}[theorem]{Remark}
\newcommand{\Gr}{\operatorname{Gr}}
\newcommand{\Fl}{\operatorname{Fl}}
\newcommand{\rk}{\operatorname{rk}}
\newcommand{\Hom}{\operatorname{Hom}}
\newcommand{\Hilb}{\operatorname{Hilb}}
\newcommand{\Pic}{\operatorname{Pic}}

\newcommand{\OO}{\mathcal O}
\newcommand{\PP}{\mathbf P}
\newcommand{\CC}{\mathbf C}
\newcommand{\cA}{\mathcal A}
\newcommand{\cW}{\mathcal W}
\newcommand{\cP}{\mathcal P}
\newcommand{\cR}{\mathcal R}
\newcommand{\cQ}{\mathcal Q}
\newcommand{\cU}{\mathcal U}
\newcommand{\Rmin}{R_2^{\mathrm{min}}}
\newcommand{\bir}{\dashrightarrow}
\DeclareMathOperator{\codim}{codim}
\title{Moduli of Conics on General Pl\"ucker Linear Sections\\of Grassmannians}
\author{Zheyuan Fu\thanks{\raggedright School of Mathematics, Shandong University, Jinan 250100, P. R. China. Email: \href{mailto:202511781@mail.sdu.edu.cn}{202511781@mail.sdu.edu.cn}.}}
\date{September 2026}
\hypersetup{pdftitle={Moduli of Conics on General Plucker Linear Sections of Grassmannians},pdfauthor={Zheyuan Fu}}
\begin{document}
\maketitle
\begin{abstract}
Let $G=\Gr(k,V)$ be a complex Grassmannian in its Pl\"ucker embedding, and let $Y_E$ be a general codimension-$r$ linear section. We study the open Hilbert scheme $R_2(Y_E)$ of smooth conics using the kernel and span of a conic. These define a flag variety $B=\Fl(k-2,k+2;V)$ and a relative Grassmannian of three-planes in a rank-six bundle. For $0\le r\le3$, we prove that $R_2(Y_E)$ is nonempty, smooth, irreducible and rational of dimension $e=2\dim V+k(\dim V-k)-3-3r$, with birational model $B\times\Gr(3,6-r)$. For $r\ge4$, a general-position argument gives a componentwise birational correspondence with a rank-three degeneracy locus and two smooth projective incidence resolutions. Every component meets the locus of smooth minimal-envelope conics; no additional density assumption is needed for general $E$. If $0\le e\le r$, the degeneracy locus itself is smooth. In dimension zero we obtain a top Chern class formula, recovering the known counts $1225$ and $1176$ by exact localization.
\end{abstract}

\section{Introduction}
Let $V$ be an $n$-dimensional complex vector space and set
\[
G=\Gr(k,V)\hookrightarrow\PP(\wedge^k V),\qquad 2\le k\le n-2.
\]
Our Grassmannians parametrize subspaces. Write $H=c_1(\OO_G(1))$ and $W_0=H^0(G,\OO_G(1))=\wedge^kV^\vee$. For an $r$-plane $E\subset W_0$, put
\[
Y_E=G\cap\PP(E^\perp),\qquad e=e(k,n,r)=2n+k(n-k)-3-3r.
\]
We assume $0\le r\le\dim G$; a general $Y_E$ is then smooth of dimension $k(n-k)-r$. The notation $R_d(Y_E)$ means the open Hilbert scheme of smooth embedded rational curves of Pl\"ucker degree $d$. In particular, $R_2(Y_E)$ does not include singular conics or double covers of lines. Throughout, ``general'' means in a nonempty Zariski open subset of the indicated parameter space.

The kernel and span of a curve in a Grassmannian are the intersection and sum of its corresponding vector subspaces. Their role in quantum Schubert calculus is classical: Buch--Kresch--Tamvakis \cite[Section 2.1]{BKT} use the bounds $\dim\ker(C)\ge k-d$ and $\dim\operatorname{span}(C)\le k+d$ for degree-$d$ rational curves. For conics, the open locus where both bounds are equalities naturally reduces to plane sections of $\Gr(2,4)$. The associated relative Grassmannian model is already explicit for $\Gr(2,5)$ in Chung--Hong--Lee \cite{CHL} and Chung--Lee \cite{CL}; the latter also describe blow-up/down constructions for Hilbert schemes of conics on quintic del Pezzo varieties. Further flip constructions for quadrics on del Pezzo varieties, including linear sections of $\Gr(2,5)$, are developed in \cite{Shah}. Related compactified spaces are studied in \cite{CM}, while Str\o mme's work \cite{Stromme} supplies the irreducibility of spaces of parametrized maps.

The purpose here is to formulate this construction uniformly for every $k,n$ and to control all the strata that could contribute extra components after imposing linear equations. There are two distinct issues: the rank of the equations on the rank-six envelope bundle may drop, and the conic itself may have a larger kernel or a smaller span. The main new point is therefore not only the relative Grassmannian description of the minimal-envelope locus, but the uniform control of its complement after imposing general linear equations. This rules out hidden components supported on nonminimal conics or on lower-rank boundary strata and leads to componentwise birational models in every codimension. Uniform estimates for both phenomena yield the following theorem.

\begin{theorem}\label{thm:low}
For $0\le r\le3$ and general $E$, the scheme $R_2(Y_E)$ is nonempty, smooth, irreducible and rational of dimension $e(k,n,r)$. More precisely,
\[
R_2(Y_E)\bir\Fl(k-2,k+2;V)\times\Gr(3,6-r)
\]
is a birational equivalence.
\end{theorem}

To state the result in higher codimension, let
\[
B=\Fl(k-2,k+2;V),\qquad
\cP=(\det\cA)\otimes\wedge^2(\cW/\cA),
\]
where $\cA\subset\cW\subset V\otimes\OO_B$ is the universal flag. The bundle $\cP$ has rank six and is a subbundle of $\wedge^kV\otimes\OO_B$. Thus $E$ induces
\[
\Phi_E:\cP\to E^\vee\otimes\OO_B.
\]
Let $S=\Gr_B(3,\cP)$ with tautological subbundle $\cR$, and let
\[
S_E=\{(b,L)\in S:L\subset\ker\Phi_{E,b}\}.
\]
We always endow rank loci with their determinantal scheme structures.

\begin{theorem}\label{thm:high}
Let $r\ge4$ and let $E$ be general. The schemes $R_2(Y_E)$, $S_E$, and $D_3(\Phi_E)$ are either all empty or all nonempty. In the latter case:
\begin{enumerate}[label=\textup{(\roman*)},leftmargin=*]
\item $R_2(Y_E)$ and $S_E$ are smooth of pure dimension $e$, and $D_3(\Phi_E)$ is reduced and Cohen--Macaulay of pure codimension $3(r-3)$ in $B$.
\item The irreducible components of these three schemes are naturally in bijection. Corresponding components are birational, and each component of $S_E$ gives a projective resolution of the corresponding component of $D_3(\Phi_E)$.
\item On every component the correspondence is induced by a dense open set of smooth conics with kernel dimension $k-2$, span dimension $k+2$, and $\rk\Phi_E=3$.
\item The dual incidence scheme
\[
\widetilde D_E=\{(b,K)\in B\times\Gr(r-3,E):K\subset\ker(\Phi_E^\vee)_b\}
\]
is also smooth, projective, and pure of dimension $e$. Its components give resolutions of the same components of $D_3(\Phi_E)$.
\end{enumerate}
If $e<0$, all these schemes are empty. If $0\le e\le r$, then $D_2(\Phi_E)=\varnothing$, and both incidence projections are isomorphisms onto the smooth scheme $D_3(\Phi_E)$.
\end{theorem}

The restriction to general $E$ is important in this statement. The density assertions are consequences of dimension estimates, rather than hypotheses imposed on selected components. The theorem does not assert irreducibility or rationality in higher codimension: the quintic del Pezzo surface has five conic components, and the genus-eight Fano threefold has a nonrational conic surface (Section~\ref{sec:examples}).

In expected dimension zero, the incidence model gives
\[
\#R_2(Y_E)=\int_S c_3(\cR^\vee)^r.
\]
We prove that the zeros avoid the boundary and recover $1225$ conics for $(k,n,r)=(2,7,7)$ and $1176$ for $(3,6,6)$. These are known numbers, recorded in \cite[Sections 7.4--7.5]{BCFKvS}. That paper also reports their verification by Str\o mme using classical methods and the Schubert package. Our contribution in these examples is the uniform incidence interpretation and the explicit boundary argument; no priority claim is made for the numerical values.

\section{General position and deformation theory}\label{sec:general}
We first record the form of general position needed below. The proof uses ordered sections so that the universal rank conditions are ordinary matrix determinantal conditions.

\begin{lemma}\label{lem:rank}
Let $X$ be a smooth irreducible complex variety, $F$ a vector bundle of rank $m$, and $W\subset H^0(X,F)$ a finite-dimensional generating subspace. Fix $0\le r\le\dim W$. For general $E\in\Gr(r,W)$, let
\[
\epsilon_E:E\otimes\OO_X\to F,\qquad
D_q(E)=\{x:\rk\epsilon_{E,x}\le q\}.
\]
For every $0\le q<\min(r,m)$, $D_q(E)$ is either empty or reduced and Cohen--Macaulay of pure codimension $(r-q)(m-q)$. Each exact-rank stratum is smooth or empty. In particular, the exact-rank-$q$ locus is dense in every component of a nonempty $D_q(E)$.
\end{lemma}
\begin{proof}
Put $M=\Hom(\CC^r,W)$. Evaluation gives a surjective morphism of vector bundles on $X$,
\[
M\otimes\OO_X\to\Hom(\CC^r,F).
\]
The induced map of total spaces is smooth. Locally on $X$ it is a projection of affine bundles. Hence the universal locus
\[
\mathscr D_q=\{(x,t)\in X\times M:\rk(t(x))\le q\}
\]
is locally the product of an open subset of $X$, an affine space, and the variety of $m\times r$ matrices of rank at most $q$. It is integral and Cohen--Macaulay of codimension $c_q=(r-q)(m-q)$; its exact-rank strata are smooth. We use here the standard determinantal facts recalled in \cite{Fulton,Manivel}.

If $\mathscr D_q\to M$ is not dominant, its general fiber is empty. Otherwise generic flatness permits us to restrict to an open subset of $M$ over which the map is flat. Its fibers are Cohen--Macaulay, since the source is Cohen--Macaulay and the base is smooth. They are pure of dimension $\dim X-c_q$: locally at a closed point, the parameters on the base form a regular sequence in the equidimensional Cohen--Macaulay source. Applying generic smoothness separately to each universal exact-rank stratum makes every nonempty stratum of a general fiber smooth.

For $q=0$, the whole fiber is an exact-rank stratum and hence is smooth. For $q\ge1$, the same dimension calculation for $q-1$ shows that its locus has smaller dimension than $\dim X-c_q$, or is empty. Thus every component of the fiber meets the smooth exact-rank-$q$ locus. A Cohen--Macaulay scheme has no embedded associated primes, so generic reducedness implies reducedness. There are only finitely many rank strata, and we may choose the parameter open set simultaneously for all of them.

Finally, restrict to injective maps $t:\CC^r\hookrightarrow W$. Their image subspaces are the points of $\Gr(r,W)$, and the properties just obtained are invariant under change of basis of $\CC^r$. They therefore hold for general $E$.
\end{proof}

\begin{lemma}\label{lem:bad}
With $X,F,W$ as above, the common zero scheme $Z_E$ of $r$ general sections is smooth of pure codimension $rm$, or empty. For a fixed closed subset $Z\subset X$,
\[
\dim(Z\cap Z_E)\le\dim Z-rm
\]
for general $E$, where a negative right-hand side means emptiness.
\end{lemma}
\begin{proof}
On $X\times W^r$, the universal zero scheme is the kernel bundle of the evaluation map to $F^{\oplus r}$. It is smooth. Generic smoothness for its projection to $W^r$ gives the first assertion. Over $Z$, the same kernel bundle has dimension $\dim Z+r\dim W-rm$. Apply the general-fiber dimension theorem to each component of its projection to $W^r$; nondominating components have empty general fiber. Restriction to independent tuples gives the assertion for general $r$-planes.
\end{proof}

For the universal sequence on $G$, write
\[
0\longrightarrow\cU\longrightarrow V\otimes\OO_G\longrightarrow\mathcal Q_G\longrightarrow0.
\]
Then $T_G=\cU^\vee\otimes\mathcal Q_G$ is globally generated, $\dim G=k(n-k)$, and $-K_G=nH$. For a smooth rational curve $C\subset X$ in a smooth variety $X$, deformation theory and Riemann--Roch give
\begin{equation}\label{eq:chi}
T_{[C]}\Hilb(X)=H^0(C,N_{C/X}),\qquad
\chi(N_{C/X})=-K_X\cdot C+\dim X-3.
\end{equation}
Obstructions lie in $H^1(C,N_{C/X})$; see \cite{Hartshorne,Kollar}. By adjunction, a degree-$d$ smooth rational curve on a smooth codimension-$r$ linear section satisfies
\[
\chi(N_{C/Y_E})=nd+k(n-k)-3-r(d+1).
\]

\begin{lemma}\label{lem:ambient}
The scheme $R_2(G)$ is smooth and irreducible of dimension $2n+k(n-k)-3$.
\end{lemma}
\begin{proof}
For a parametrization $f:\PP^1\to C\subset G$, the bundle $N_{C/G}$ is a quotient of the globally generated bundle $f^*T_G$. It is globally generated, so its first cohomology vanishes. Equation~\eqref{eq:chi} gives the dimension and smoothness.

A map to our subspace Grassmannian corresponds to a locally free quotient $V^\vee\otimes\OO_{\PP^1}\twoheadrightarrow f^*\cU^\vee$ of rank $k$ and degree two. The corresponding Quot compactification is irreducible by \cite{Stromme}. Its nonempty open subset of embedded maps is therefore irreducible, and its image in the Hilbert scheme is precisely $R_2(G)$. This proves irreducibility.
\end{proof}

Let $p:\mathscr C_2\to R_2(G)$ and $q:\mathscr C_2\to G$ be the universal conic and its evaluation map. Set $\mathcal E_2=p_*q^*\OO_G(1)$.
\begin{proposition}\label{prop:smooth}
The bundle $\mathcal E_2$ has rank three and is generated by $W_0$. The zero scheme of the sections induced by $E$ is $R_2(Y_E)$ scheme-theoretically. For general $E$, this scheme is smooth of pure dimension $e(k,n,r)$, or empty; it is empty if $e(k,n,r)<0$.
\end{proposition}
\begin{proof}
On a smooth conic, $\OO_G(1)|_C=\OO_{\PP^1}(2)$, which has three sections and vanishing first cohomology. Cohomology and base change give local freeness. The Pl\"ucker span of $C$ is a plane, and restriction of linear forms to it and then to $C$ is surjective. Thus $W_0$ generates $\mathcal E_2$.

For a test scheme $T\to R_2(G)$, base change identifies $\mathcal E_2|_T$ with the pushforward of $\OO_G(1)$ on its family $\mathscr C_T$. The induced section of this pushforward is zero exactly when the original section vanishes on $\mathscr C_T$, that is, when $\mathscr C_T$ is contained scheme-theoretically in the corresponding hyperplane. This functorial criterion for every element of a basis of $E$ proves the scheme-theoretic equality. Lemma~\ref{lem:bad} now applies.
\end{proof}

\section{Kernel--span flags and exceptional conics}\label{sec:flags}
For a smooth conic $C\subset G$, put
\[
A_C=\bigcap_{U\in C}U,\qquad W_C=\sum_{U\in C}U.
\]
The splitting of the tautological bundle gives $\dim A_C\ge k-2$ and $\dim W_C\le k+2$, in agreement with the general kernel--span bounds in \cite{BKT}.

\begin{definition}
A smooth conic has \emph{minimal flag envelope} if $\dim A_C=k-2$ and $\dim W_C=k+2$. Write $\Rmin(G)$ for this open locus. Here ``minimal'' refers to the uniquely determined enclosing flag of the indicated dimensions; equivalently, this is the locus where both kernel--span bounds are attained.
\end{definition}

Recall $B$, $\cP$ and $S$ from the introduction, and put $\cQ=\cW/\cA$. Exterior multiplication gives the bundle inclusion
\[
\cP=(\det\cA)\otimes\wedge^2\cQ\hookrightarrow\wedge^kV\otimes\OO_B.
\]
For $b=(A,W)\in B$, the sub-Grassmannian of $k$-planes between $A$ and $W$ is
\[
\Gr(2,W/A)\simeq Q_b^4\subset\PP(\cP_b)\simeq\PP^5.
\]
A point $(b,L)\in S$ thus determines a plane $\PP(L)\subset\PP(\cP_b)$. Let $S^\circ\subset S$ be the locus where this plane cuts $Q_b^4$ in a smooth conic with minimal flag envelope.

\begin{proposition}\label{prop:model}
There is an isomorphism $S^\circ\simeq\Rmin(G)$. Both are dense open subsets of their respective irreducible ambient spaces, and
\[
\dim B=k(n-k)+2n-12,\qquad\dim S=2n+k(n-k)-3.
\]
\end{proposition}
\begin{proof}
The plane-section construction gives a flat family of smooth conics on $S^\circ$. This open set is nonempty. For example, in a four-dimensional quotient with basis $e_1,e_2,e_3,e_4$, the planes
\[
U_{[s:t]}=\langle se_1+te_3,\ se_2+te_4\rangle
\]
have zero common intersection, span the quotient, and trace an embedded Pl\"ucker conic. Adding a fixed $(k-2)$-plane supplies an example in $G$.

We justify the inverse in families. On the universal minimal conic, let $u:\mathscr C^{\mathrm{min}}\to\Rmin(G)$ be the projection and $v$ the map to $G$. On every fiber,
\[
v^*\cU|_C\simeq\OO_{\PP^1}(-1)^{\oplus2}\oplus\OO_{\PP^1}^{\oplus(k-2)}.
\]
Indeed, this is forced by the degree $-2$, nonpositive splitting, and the equality $h^0(v^*\cU|_C)=k-2$. Base change shows that $u_*v^*\cU$ is a rank-$(k-2)$ subbundle of $V\otimes\OO$ with fibers $A_C$. Similarly, $u_*\bigl((v^*\cU)^\vee\bigr)$ is locally free of rank $k+2$. The map
\[
V^\vee\otimes\OO\to u_*\bigl((v^*\cU)^\vee\bigr)
\]
has kernel $W_C^\perp$ on each fiber and is surjective, since its image has dimension $\dim W_C=k+2$. Its kernel therefore determines the subbundle with fibers $W_C$. Finally, dualizing the surjection
\[
\wedge^kV^\vee\otimes\OO\to u_*v^*\OO_G(1)
\]
gives the rank-three bundle of Pl\"ucker spans $L_C$.

These bundles define a morphism to $S$. For a minimal conic, its span plane cannot be contained in $Q_b^4$: a plane contained in $\Gr(2,4)$ is one of the two standard plane types, forcing either a fixed line or a three-dimensional upper span in $W/A$. Hence the nonzero quadratic equation of $Q_b^4$ on the span plane cuts out precisely $C$. The constructions are inverse in families. The same two-plane classification is detailed in Lemma~\ref{lem:types} below. The dimension formulas follow from the flag and Grassmannian dimensions, and density in $R_2(G)$ follows from Lemma~\ref{lem:ambient}.
\end{proof}

\begin{definition}
An $\alpha$-plane and a $\beta$-plane in $G$ are respectively
\[
\begin{aligned}
\Pi^\alpha_{A',W}&=\{U:A'\subset U\subset W\},
&&\dim A'=k-1,\quad\dim W=k+2,\\
\Pi^\beta_{A,W'}&=\{U:A\subset U\subset W'\},
&&\dim A=k-2,\quad\dim W'=k+1.
\end{aligned}
\]
Each is a linearly embedded projective plane. Denote the loci of smooth conics spanning these planes by $\Sigma_\alpha$ and $\Sigma_\beta$.
\end{definition}

\begin{lemma}\label{lem:types}
Every nonminimal smooth conic lies in exactly one of $\Sigma_\alpha,\Sigma_\beta$. Moreover,
\[
\codim_{R_2(G)}\Sigma_\alpha=k-1,\qquad
\codim_{R_2(G)}\Sigma_\beta=n-k-1.
\]
For general $E$ of any allowed dimension $r$,
\begin{equation}\label{eq:special}
\begin{split}
\dim(\Sigma_\alpha\cap R_2(Y_E))&\le e-(k-1),\\
\dim(\Sigma_\beta\cap R_2(Y_E))&\le e-(n-k-1).
\end{split}
\end{equation}
Negative upper bounds mean emptiness.
\end{lemma}
\begin{proof}
For a parametrization $f:\PP^1\to C\subset G$, the possible splittings are
\[
f^*\cU=\OO(-2)\oplus\OO^{\oplus(k-1)}
\quad\hbox{or}\quad
f^*\cU=\OO(-1)^{\oplus2}\oplus\OO^{\oplus(k-2)}.
\]
In the first case the trivial summands give a fixed $(k-1)$-plane $A'$. The quotient line varies by quadratic polynomials, and an embedding of degree two has a three-dimensional span in $V/A'$. Thus $C$ spans the unique $\alpha$-plane determined by $A'$ and $W_C$.

In the second case, the common intersection is the fixed $(k-2)$-plane $A$. The map from $\OO(-1)^{\oplus2}$ into the trivial quotient bundle has upper span of dimension at most four. If this span has dimension four, the conic is minimal. It cannot have dimension two, since the induced map to $\Gr(2,2)$ would be constant. If its dimension is three, $C$ spans the $\beta$-plane $\Gr(2,W_C/A)$. Conversely, conics in the two stated plane types have respectively kernel dimension $k-1$ and upper span dimension $k+1$, so are nonminimal. The classical classification of planes in a Grassmannian says that every contained plane has one of these two forms; fiberwise it is the two families of planes on $\Gr(2,4)=Q^4$.

The plane parameter spaces are
\[
F_\alpha=\Fl(k-1,k+2;V),\qquad
F_\beta=\Fl(k-2,k+1;V).
\]
Over either space the smooth conics form the open subset of a projective bundle with fiber $\PP^5$ parametrizing nondegenerate quadratic equations. Subtracting these dimensions from $2n+k(n-k)-3$ gives the claimed codimensions.

A Pl\"ucker hyperplane contains a nondegenerate conic in a plane if and only if it contains the entire plane. On each $F_\alpha,F_\beta$, the dual of the rank-three span bundle is generated by $W_0$. Lemma~\ref{lem:bad} makes the common zero scheme of $r$ general sections empty or of codimension $3r$. Adding the five-dimensional family of conics in each surviving plane proves \eqref{eq:special}.
\end{proof}

The section induced by $E$ of $\cR^\vee\otimes E^\vee$ has zero scheme $S_E$. On $S^\circ$, Proposition~\ref{prop:model} identifies it scheme-theoretically with $\Rmin(Y_E)$.

\begin{proposition}\label{prop:boundary}
For general $E$, $S_E$ is smooth of pure dimension $e$, or empty, and
\[
\dim(S_E\setminus S^\circ)\le e-1.
\]
The loci $\Rmin(Y_E)$ are dense in every irreducible component of both $S_E$ and $R_2(Y_E)$ under their open identifications. These two schemes are simultaneously empty, or their irreducible components are naturally in bijection and corresponding components are birational.
\end{proposition}
\begin{proof}
The bundle $\cR^\vee$ is generated by $W_0$, because it is a quotient of the pullback of $\cP^\vee$. Lemma~\ref{lem:bad} gives smoothness and pure dimension. The same lemma applied to the proper closed subset $S\setminus S^\circ$, of dimension at most $\dim S-1$, gives the boundary estimate. No component of $S_E$ is contained in that boundary.

By Proposition~\ref{prop:smooth}, every component of $R_2(Y_E)$ has dimension $e$. Its nonminimal locus is the union of the two loci in \eqref{eq:special}, each of dimension strictly smaller than $e$. Thus every component meets the minimal locus. The shared open scheme $S_E\cap S^\circ=\Rmin(Y_E)$ is dense in every component on either side. Components of a smooth scheme are disjoint, so this shared open scheme identifies their components bijectively and preserves their function fields. The argument also proves the emptiness equivalence.
\end{proof}

\section{The uniform rationality range}\label{sec:low}
\begin{proof}[Proof of Theorem~\ref{thm:low}]
For $r=0$, the result follows from Proposition~\ref{prop:model} and the rationality of the relative Grassmannian $S$. Assume $1\le r\le3$. Apply Lemma~\ref{lem:rank} to the map $E\otimes\OO_B\to\cP^\vee$. The locus $D_{r-1}$ is proper or empty, so the full-rank open set $B^{\mathrm{fr}}=B\setminus D_{r-1}$ is nonempty. On this open set, $\mathcal K=\ker\Phi_E$ is locally free of rank $6-r\ge3$, and
\[
S_E|_{B^{\mathrm{fr}}}=\Gr_{B^{\mathrm{fr}}}(3,\mathcal K)
\]
is nonempty and irreducible. Its dimension is $\dim B+3(3-r)=e$.

On the exact-rank-$q$ stratum, $q<r$, the fiber of $S_E\to B$ is $\Gr(3,6-q)$ of dimension $3(3-q)$. Lemma~\ref{lem:rank} gives
\begin{align*}
\dim\bigl(S_E|_{\rk\Phi_E=q}\bigr)
&\le\dim B-(6-q)(r-q)+3(3-q)\\
&=e+(r-q)(q-3)<e.
\end{align*}
By Proposition~\ref{prop:boundary}, $S_E$ is smooth and pure of dimension $e$. No component can be supported over the lower-rank strata, so its irreducible full-rank open set is dense in the whole scheme. Hence $S_E$ is irreducible.

The flag variety $B$ is rational. Trivializing $\mathcal K$ on a nonempty open subset gives
\[
S_E\bir B\times\Gr(3,6-r).
\]
Finally Proposition~\ref{prop:boundary} says that $\Rmin(Y_E)$ is nonempty and dense in $S_E$, and that $R_2(Y_E)$ has the same number of components and the same function field. Together with Proposition~\ref{prop:smooth}, this proves all assertions.
\end{proof}

\section{Degeneracy loci and incidence resolutions}\label{sec:high}
Assume $r\ge4$ and set $D_q=D_q(\Phi_E)$. A three-dimensional subspace of $\ker\Phi_{E,b}$ exists exactly when $\rk\Phi_{E,b}\le3$. Consequently the projective projection
\[
f:S_E\to B
\]
has image $D_3$ on underlying sets, and factors through its determinantal scheme structure. Indeed, on $S_E$ the map $\Phi_E$ factors through the rank-three quotient $\cP/\cR$, so all its $4\times4$ minors vanish. On
\[
D_3^\circ=D_3\setminus D_2
\]
the kernel is a rank-three vector bundle and provides a scheme-theoretic inverse $b\mapsto(b,\ker\Phi_{E,b})$.

\begin{proposition}\label{prop:resolution}
For general $E$, $D_3$ is empty or reduced and Cohen--Macaulay of pure dimension $e$. Every component of $S_E$ dominates a unique component of $D_3$ birationally, and all components of $D_3$ occur exactly once. The resulting morphisms of components are projective resolutions.
\end{proposition}
\begin{proof}
Lemma~\ref{lem:rank} gives pure codimension $3(r-3)$, reducedness and the Cohen--Macaulay property. On each nonempty exact-rank-$q$ stratum with $q<3$, the dimension calculation is
\begin{equation}\label{eq:lower}
\dim\bigl(S_E|_{\rk\Phi_E=q}\bigr)
\le\dim B-(6-q)(r-q)+3(3-q)
=e-(r-q)(3-q)<e.
\end{equation}
Thus no component of the smooth pure-dimensional scheme $S_E$ is supported over $D_2$. Every component of $D_3$ meets $D_3^\circ$ densely by Lemma~\ref{lem:rank}. Since $f$ is an isomorphism above $D_3^\circ$, its components and those of $D_3$ are in bijection, with birational component maps. The restrictions of $f$ are projective and have smooth sources, so give resolutions of the reduced components.
\end{proof}

\begin{proposition}\label{prop:dual}
For general $E$, the scheme $\widetilde D_E$ of Theorem~\ref{thm:high} is smooth, projective and pure of dimension $e$, or empty. Its projection to $D_3$ is an isomorphism over $D_3^\circ$, and its components give projective resolutions of all components of $D_3$.
\end{proposition}
\begin{proof}
Fix an $r$-dimensional vector space $H_r$. Let $\mathcal K$ denote the tautological rank-$(r-3)$ subbundle on $\Gr(r-3,H_r)$. On $B\times\Gr(r-3,H_r)$, a map $t:H_r\to W_0$ induces a section of
\[
\mathcal K^\vee\otimes\cP^\vee.
\]
The space $\Hom(H_r,W_0)$ generates this bundle: evaluation $W_0\to\cP_b^\vee$ and restriction $H_r^\vee\to K^\vee$ are both surjective. For general $t$ its zero scheme is therefore transverse by Lemma~\ref{lem:bad}. Restricting to injective $t$ and identifying $H_r$ with its image $E$ gives precisely $\widetilde D_E$. Its dimension is
\[
\dim B+3(r-3)-6(r-3)=e.
\]
Over a rank-three point the kernel of $E\to\cP_b^\vee$ has dimension $r-3$, so $K$ is uniquely determined and the map is an isomorphism there. At rank $q<3$, the fiber is $\Gr(r-3,r-q)$, and
\begin{align*}
\dim\bigl(\widetilde D_E|_{\rk\Phi_E=q}\bigr)
&\le\dim B-(6-q)(r-q)+(r-3)(3-q)\\
&=e-(3-q)(6-q)<e.
\end{align*}
No component is supported over $D_2$. The same component argument as in Proposition~\ref{prop:resolution} finishes the proof.
\end{proof}

\begin{corollary}\label{cor:small}
Suppose $r\ge4$, $E$ is general and $0\le e\le r$. Then $D_2=\varnothing$, and
\[
S_E\simeq D_3\simeq\widetilde D_E.
\]
In particular $D_3$ is a smooth projective scheme, possibly empty, whose components are birational to those of $R_2(Y_E)$. This applies in particular whenever $0\le e\le3$ and $r\ge4$.
\end{corollary}
\begin{proof}
If nonempty, $D_2$ would have dimension
\[
\dim B-4(r-2)=e-r-1<0,
\]
contrary to Lemma~\ref{lem:rank}. Both incidence projections are therefore isomorphisms everywhere. Smoothness follows from Proposition~\ref{prop:boundary} or from smoothness of the exact-rank-three stratum.
\end{proof}

\begin{proof}[Proof of Theorem~\ref{thm:high}]
The incidence fibers show that $S_E$ is empty exactly when $D_3$ is empty. Proposition~\ref{prop:boundary} gives the corresponding equivalence with $R_2(Y_E)$, the componentwise birational correspondence, and density of the smooth minimal locus. Equation~\eqref{eq:lower} shows that its intersection with the rank-three locus is still dense in every component. Combine Propositions~\ref{prop:smooth}, \ref{prop:resolution}, \ref{prop:dual} and Corollary~\ref{cor:small}. If $e<0$, Proposition~\ref{prop:boundary} makes $S_E$ empty, and the same equivalences apply.
\end{proof}

\subsection*{Alternating forms and the dual projection}
The second incidence model admits the projection
\[
\rho:\widetilde D_E\to\Gr(k-2,V)\times\Gr(r-3,E),
\quad (A,W,K)\mapsto(A,K).
\]
For $\eta\in K$, contraction with $A$ defines canonically
\[
\omega_{\eta,A}\in(\det A)^\vee\otimes\wedge^2(V/A)^\vee.
\]
After choosing a generator of $\det A$, this is an alternating form defined up to a common nonzero scalar. For $U=W/A$,
\[
\eta|_{(\det A)\otimes\wedge^2 U}=0
\quad\Longleftrightarrow\quad \omega_{\eta,A}|_U=0.
\]
The fiber of $\rho$ at $(A,K)$ is therefore the scheme of common isotropic four-planes for the forms in $K$. For $r=4$ there is one form; for $r=5$ there is a pencil. Pfaffian rank conditions for such forms are classical \cite{HT}. This description is useful for studying individual triples, but does not itself assert that the fibers are nonempty or irreducible.

\subsection*{Canonical bundle}
Let $\pi:S\to B$. The relative tangent bundle is $\cR^\vee\otimes(\pi^*\cP/\cR)$, so
\[
K_S=\pi^*K_B\otimes\pi^*(\det\cP)^{-3}\otimes(\det\cR^\vee)^{-6}.
\]
If the section defining $S_E$ is regular, $S_E$ is a local complete intersection and its dualizing line bundle is
\begin{equation}\label{eq:canonical}
\omega_{S_E}=\left[\pi^*K_B\otimes\pi^*(\det\cP)^{-3}
\otimes(\det\cR^\vee)^{r-6}\otimes(\det E^\vee)^3\right]_{S_E}.
\end{equation}
For a general $E$ the section is transverse, so $S_E$ is smooth and this is its canonical bundle. Regularity alone does not imply smoothness.

For an explicit divisor-class form, let $h_A$ and $h_W$ denote the classes of $\det\cA^\vee$ and $\det\cW^\vee$ in $\Pic(B)$, and let $h_R$ denote the class of $\det\cR^\vee$ in $\Pic(S)$; we use additive notation. A filtration of $T_B$ has factors $\Hom(\cA,\cW/\cA)$ and $\Hom(\cW,V/\cW)$; hence, up to constant line bundles,
\[
K_B=(\det\cA)^{k+2}\otimes(\det\cW)^{n-k+2},
\qquad \det\cP=(\det\cA)^3\otimes(\det\cW)^3.
\]
Consequently \eqref{eq:canonical} becomes the equality in $\Pic(S_E)$
\begin{equation}\label{eq:canonicalclass}
K_{S_E}=(7-k)\pi^*h_A+(k+7-n)\pi^*h_W+(r-6)h_R.
\end{equation}
Here and in this formula we identify line bundles with their divisor classes. If $k=2$ or $k+2=n$, the corresponding constant determinant class is zero. Under Corollary~\ref{cor:small}, formula~\eqref{eq:canonicalclass} also computes the canonical class of $D_3$ via its identification with $S_E$. Its positivity requires understanding the restrictions of these classes, and is not determined by the signs of the displayed coefficients alone.

\section{Zero-dimensional conic schemes and localization}\label{sec:counts}
\begin{theorem}\label{thm:count}
Suppose $e(k,n,r)=0$ and $E$ is general. Then $R_2(Y_E)$ is a finite reduced scheme, possibly empty, and
\begin{equation}\label{eq:count}
\#R_2(Y_E)=\int_S c_3(\cR^\vee)^r.
\end{equation}
\end{theorem}
\begin{proof}
The bundle $(\cR^\vee)^{\oplus r}$ has rank $3r=\dim S$, and a general defining section is transverse. Thus its zero scheme $S_E$ is finite and reduced, with length equal to the integral in \eqref{eq:count}. Proposition~\ref{prop:boundary} gives $S_E\setminus S^\circ=\varnothing$, since that locus has dimension at most $-1$. Likewise, \eqref{eq:special} excludes all nonminimal conics. Hence the open scheme identification is in fact an equality $S_E\simeq R_2(Y_E)$ of finite reduced schemes. This identification with the complete zero scheme ensures that its top Chern number counts exactly the smooth embedded conics, with no boundary contribution.
\end{proof}

Let the diagonal torus of $\operatorname{GL}(V)$ have independent characters $\lambda_1,\ldots,\lambda_n$. A fixed flag in $B$ corresponds to
\[
I\subset J\subset\{1,\ldots,n\},\qquad |I|=k-2,\quad |J|=k+2.
\]
Set $Q=J\setminus I$ and let $\mathcal U_Q$ be the set of two-element subsets of $Q$. The six weights of $\cP$ at this flag are
\[
p_u=\sum_{i\in I}\lambda_i+\lambda_a+\lambda_b,
\qquad u=\{a,b\}\in\mathcal U_Q.
\]
A fixed point of $S$ is a flag together with a three-element subset $T\subset\mathcal U_Q$. Define
\begin{align*}
\Delta_{I,J}&=\prod_{i\in I}\prod_{a\in J\setminus I}(\lambda_a-\lambda_i)
\prod_{i\in J}\prod_{b\notin J}(\lambda_b-\lambda_i),\\
\Delta_T&=\prod_{u\in T}\prod_{v\in\mathcal U_Q\setminus T}(p_v-p_u).
\end{align*}
These are the Euler classes of the flag tangent space and the relative Grassmannian tangent space. Equivariant localization \cite{Brion,Fulton} yields
\begin{equation}\label{eq:localization}
\int_S c_3(\cR^\vee)^r
=\sum_{\substack{I\subset J\\|I|=k-2,\ |J|=k+2}}
\ \sum_{\substack{T\subset\mathcal U_Q\\|T|=3}}
\frac{\left(\prod_{u\in T}(-p_u)\right)^r}{\Delta_{I,J}\Delta_T}.
\end{equation}
The use of pair indices rather than just numerical weight values keeps the indexing unambiguous. Since $S$ is proper and the integrand has codimension $\dim S$, its equivariant pushforward lies in $A_T^0(\mathrm{pt})_{\mathbb Q}=\mathbb Q$. The localization sum therefore represents a constant, independent of every specialization for which the denominators are nonzero. Homogeneity of degree zero alone would not imply this independence.

\begin{corollary}\label{cor:counts}
For general linear sections,
\[
\begin{split}
\#R_2\bigl(\Gr(2,7)\cap H_1\cap\cdots\cap H_7\bigr)&=1225,\\
\#R_2\bigl(\Gr(3,6)\cap H_1\cap\cdots\cap H_6\bigr)&=1176.
\end{split}
\]
\end{corollary}
\begin{proof}
In the first case $B=\Gr(4,7)$ and $\cP=\wedge^2\cW$. Formula~\eqref{eq:localization} has $\binom74\binom63=700$ summands. In the second case $B=\Fl(1,5;6)$ and $\cP=\cA\otimes\wedge^2(\cW/\cA)$; there are $6\cdot5\cdot\binom63=600$ summands. Exact rational evaluation gives the two asserted integers. For example, one may specialize to $\lambda_i=2^{i-1}$; distinct pair sums guarantee nonzero relative denominators, and distinct $\lambda_i$ guarantee nonzero base denominators. A second specialization $\lambda_i=3^{i-1}$ gives the same values. Theorem~\ref{thm:count} identifies these lengths with numbers of smooth embedded conics.
\end{proof}

\begin{remark}
For $(k,n,r)=(2,4,3)$ the sum is $1$, as expected: a general plane cuts $\Gr(2,4)=Q^4\subset\PP^5$ in one smooth conic. The two threefolds in Corollary~\ref{cor:counts} are linear-section Calabi--Yau examples in \cite[Table 1, Nos. 12 and 19]{IIM}; the first also occurs in the Pfaffian--Grassmannian setting \cite{Rodland}. The counts agree with the degree-two instanton numbers in \cite{BCFKvS}. We do not identify an uncorrected degree-two stable-map invariant with the Hilbert-scheme count: multiple covers of lines belong to the former problem, whereas the present calculation counts embedded conics directly.
\end{remark}

\section{Examples and remaining questions}\label{sec:examples}
\subsection*{Five components in codimension four}
For
\[
Y=\Gr(2,5)\cap H_1\cap\cdots\cap H_4,
\]
the quintic del Pezzo surface, one has \cite[Lemma 3.2]{CHL}
\[
R_2(Y)\simeq\bigsqcup_{i=1}^5\bigl(\PP^1\setminus\{0,1,\infty\}\bigr).
\]
Thus a uniform irreducibility statement cannot extend from $r\le3$ to $r=4$.

The dual incidence model makes its projective components explicit. Here $A=0$, and a point of $\widetilde D_E$ is a pair consisting of a line $K\subset E$ of alternating forms on $\CC^5$ and an isotropic four-plane $W$. A rank-four alternating form on $\CC^5$ has no isotropic four-plane, whereas a rank-two form has a three-dimensional radical and its isotropic four-planes are parametrized by $\PP^1$. For general $E$, the projective space $\PP(E)\simeq\PP^3$ meets the rank-two locus $\Gr(2,V^\vee)\subset\PP(\wedge^2V^\vee)$ transversely in five points, its degree. The five fibers are copies of $\PP^1$. Since $e=1\le r$, Corollary~\ref{cor:small} gives
\[
S_E\simeq D_3\simeq\widetilde D_E\simeq\bigsqcup_{i=1}^5\PP^1.
\]
This agrees with the compactifications of the five conic pencils on $Y$.

\subsection*{A nonrational surface in codimension five}
For
\[
X=\Gr(2,6)\cap H_1\cap\cdots\cap H_5,
\]
a general prime Fano threefold of genus eight (see also the survey \cite{KPS}), the Hilbert scheme of conics is isomorphic to the Fano surface of lines on the orthogonal cubic threefold \cite{IM}. It is a smooth irreducible surface of general type with invariants \cite[Theorem 1.1.1, case (1.8)]{KPS}
\[
p_g=10,\qquad q=5,\qquad K^2=45.
\]
The smooth-conic locus is dense \cite[Section 4.3]{IM}, so $R_2(X)$ has the same function field and is nonrational. Here $e=2$, so our $D_3$ is also a smooth projective surface birational to this Fano surface. We assert birationality of these models, not an identification of their compactification boundaries. These examples separate the possible failures of irreducibility and rationality beyond the uniform range.

\subsection*{Questions}
The remaining issues for general sections concern the geometry of the components themselves. For which triples $(k,n,r)$ with $r\ge4$ is $D_3$ nonempty and irreducible? When are its components rational or stably rational? Formula~\eqref{eq:canonicalclass}, combined with intersection calculations and information about the nef cone, may help distinguish their birational types. In dimensions at most three, Corollary~\ref{cor:small} removes the need to resolve the degeneracy locus first.

A second question concerns families. Over a nonempty open subset of $\Gr(r,W_0)$ where the incidence spaces are smooth, their components define a finite component cover because the incidence morphisms are proper and smooth. What monodromy acts on this cover, and how does it compare with the Stein factorization of the relative Hilbert scheme of all conics? For special $E$, the general-position bounds may fail, and nonminimal or lower-rank components require a separate analysis. For general $E$, such components have already been excluded by Propositions~\ref{prop:boundary} and \ref{prop:resolution}.

\enlargethispage{\baselineskip}
\section*{Statements and Declarations}
\textbf{Funding.} This research received no financial support.

\noindent
\textbf{Competing interests.} The author has no relevant financial or non-financial interests to disclose.

\noindent\textbf{Data availability.} Data sharing is not applicable to this article as no datasets were generated or analysed during the current study.

\end{document}